\documentclass{elsarticle}
\usepackage{amsmath,amssymb,amsthm,amsfonts}
\usepackage{enumerate,enumitem}
\usepackage{hyperref}
\usepackage{fullpage}
\usepackage{float}
\usepackage{setspace}
\usepackage{epsfig, graphicx}
\usepackage{subfig}
\usepackage{xcolor}
\usepackage{url}

\newtheorem{theorem}{Theorem}[section]
\newtheorem{lemma}[theorem]{Lemma}

\newtheorem{observation}[theorem]{Observation}

\newtheorem{corollary}[theorem]{Corollary}

\def\+{{}^+\!}
\def\eg{\widehat{g}}
\def\egp{\widehat{g}\+}
\def\ng{\widetilde{g}}
\def\ngp{\widetilde{g}\+}
\def\Forb{\text{Forb}}
\def\dc#1#2{\Delta_{#1}(#2)}
\def\Gxy{\G_{xy}}
\def\Gcxy{\G^\circ_{xy}}
\def\Cc{\C^\circ}
\def\hat{\widehat}
\def\os{\sigma}
\def\osp{\sigma\+}
\def\eh{\widehat{h}}
\def\nh{\widetilde{h}}
\def\nt{\widetilde{\theta}}
\def\Hw{\H^{\rm w}}

\def\et{\theta}
\def\eeta{\eta}

\def\df#1{{\emph{#1}}}

\def\A{\mathcal{A}}

\def\C{\mathcal{C}}
\def\E{\mathcal{E}}

\def\G{\mathcal{G}}
\def\H{\mathcal{H}}
\def\M{\mathcal{M}}

\def\P{\mathcal{P}}
\def\S{\mathcal{S}}

\def\NN{\mathbb{N}}
\def\RR{\mathbb{R}}
\def\SS{\mathbb{S}}

\newcounter{cases}
\newcounter{subcases}[cases]

\def\ss{\subseteq}
\def\sm{\setminus}
\def\into{\to}

\def\iso{\cong} 
\def\eg{\widehat{g}}
\def\ng{\widetilde{g}}

\begin{document}

\begin{frontmatter}

\title{Excluded minors for the Klein Bottle I. Low connectivity case}

\author[SFU]{Bojan Mohar\fnref{mohar,mohar2}}
\ead{mohar@sfu.ca}

\author[SFU]{Petr \v{S}koda}

\address[SFU]{Simon Fraser University\\
 Department of Mathematics\\
 8888 University Drive\\
 Burnaby, BC, Canada\\}

\fntext[mohar]{Supported in part by an NSERC Discovery Grant R611450 (Canada),
   by the Canada Research Chair program, and by the
   Research Grant J1-8130 of ARRS (Slovenia).}
\fntext[mohar2]{On leave from:
    IMFM \& FMF, Department of Mathematics, University of Ljubljana, Ljubljana,
    Slovenia.}

\begin{abstract}
Graphs that are critical (minimal excluded minors) for embeddability in surfaces are studied. In Part I we consider the structure of graphs with a 2-vertex-cut that are critical with respect to the Euler genus. A general theorem describing the building blocks is presented. These constituents, called hoppers and cascades, are classified for the case when Euler genus is small. As a consequence, the complete list of obstructions of connectivity 2 for embedding graphs into the Klein bottle is obtained. This is the first complete result about obstructions for embeddability of graphs in the Klein bottle, and the outcome is somewhat surprising in the sense that there are considerably fewer excluded minors than expected.
\end{abstract}


\end{frontmatter}



\section{Introduction}

Robertson and Seymour~\cite{robertson-1990} proved that for each surface $\SS$
the class of graphs that embed into $\SS$ can be characterized by a finite list $\Forb(\SS)$ of minimal forbidden minors (or \df{obstructions}).
For the 2-sphere $\SS_0$, $\Forb(\SS_0)$ consists of the \df{Kuratowski graphs}, $K_5$ and $K_{3,3}$ \cite{kuratowski-1930}. The list of obstructions $\Forb(\NN_1)$
for the projective plane $\NN_1$ already contains 35 graphs and $\NN_1$ is the only other surface for which the complete list of forbidden minors is known \cite{archdeacon-1981,glover-1979}.
For the torus $\SS_1$, the complete list of obstructions is still not known, but thousands of obstructions
were generated by the use of computers (see~\cite{Chambers02,NeufeldMyrwold,Wood07}).

The obstructions for the Klein bottle are even less understood than those for the torus. Even though no list of obstructions have been constructed so far,
it is expected that the total number of obstructions for the Klein bottle will be in tens of thousand. Henry Glover (private communication to B.M.) conjectured that there will be many more. In fact, Glover made a speculation that more than $10^6$ obstructions will be obtained by pasting together two obstructions for the projective plane by identifying two vertices in all possible ways. One of the side results of this paper is a refutation of this conjecture.

In this paper, we study critical graphs for Euler genus of low connectivity.
For a graph $G$, we denote by $\eg(G)$ its Euler genus; see Section \ref{sc-preliminaries} for definitions.
A graph $G$ is \df{critical} for Euler genus $k$ if $\eg(G) > k$ and for each edge $e \in E(G)$, $\eg(G-e) \le k $ and $\eg(G / e) \le k$, where $G/e$ denotes the graph obtained from $G$ by contracting the edge $e$.
Let $\E_k$ be the class of critical graphs for Euler genus $k$ and $\E = \bigcup_{k\ge 0} \E_k$.
It is easy to show that graphs in $\E$ that are not 2-connected can be obtained as disjoint unions and 1-sums of graphs in $\E$ (see~\cite{stahl-1977}).
Here we study graphs in $\E$ of \df{connectivity} 2, that is, graphs that are 2-connected but not 3-connected.
We shall show that each critical graph for Euler genus of connectivity 2 can be obtained as a 2-sum of two graphs that are close to graphs in $\E$
or belong to an exceptional class of graphs, called \df{cascades}. Using the classification of cascades from Part II of this work \cite{MS2014_II}, we construct the complete list of critical graphs for Euler genus 2, whose connectivity is 2.
In Sect.~\ref{sc-klein-klein}, we show that a graph of connectivity 2 is critical for Euler genus 2 if and only if it is an obstruction for the Klein bottle.
This yields a complete list of obstructions for the Klein bottle of connectivity 2. The list of obstructions for embeddability in the Klein bottle (and for Euler genus 2) contains precisely 668 graphs of connectivity two.
This is in a strong contrast with the afore-mentioned predictions of Henry Glover, who estimated that the number of Klein bottle obstructions of connectivity two will be more than a million.
An analogous result for the torus is given in~\cite{mohar-torus}. However, the methods used in that paper are quite different from those in this one. The main difference is the appearance of cascades, whose treatment occupies the whole Part II of this paper~\cite{MS2014_II}.

The above-mentioned result that obstructions of connectivity two for Euler genus 2 and the nonorientable genus 2 are the same is just a coincidence. It is easy to see that it no longer holds for larger genus. Also, there are 3-connected obstructions for Euler genus 2 that are not Klein bottle obstructions. One example is the following graph. Let $Q$ be the graph obtained from $K_7$ by first subdividing two of its edges that have no vertex in common and then adding an edge joining both vertices of degree two used in the subdivision. Since $K_7$ does not embed in the Klein bottle, $Q$ is not an obstruction for this surface. However, $Q$ cannot be embedded in the torus and, as the reader may verify, deleting or contracting any edge gives a graph of genus one. So, $Q$ is an obstruction for the torus and an obstruction for Euler genus 2.

In classifying obstructions of connectivity two, we encounter two special families of graphs that are the building blocks of such obstructions. The first class are mysterious graphs called {\em hoppers}. While we prove that hoppers do not exist when the genus is small, and we are not able to construct any for larger genus, we believe that they may show up when the genus is large enough. Their existence is closely related to an old open problem dating back to the 1980's asking if there exists a graph which is simultaneously an obstruction for two different nonorientable surfaces.\footnote{This problem was proposed in various incarnations by Dan Archdeacon, Bruce Richter, and Jozef \v{S}iran, and appears as Problem \#1 in the list of open problems in topological graph theory compiled by Dan Archdeacon in 1995 (\url{http://www.emba.uvm.edu/~darchdea/problems/decgenus.htm}).}
For such an obstruction, deleting or contracting any edge would reduce the nonorientable genus by at least two.

The graphs in the second family that we encounter are called {\em cascades}. In Part II of this work, we determine all cascades when the genus is small. The proofs use methods from structural graph theory and involve development of results about extensions of embeddings of subgraphs. The classification of cascades for the Klein bottle and for Euler genus 2 is the most complicated part of the paper.

In Part I, obstructions of connectivity two for arbitrary Euler genus are examined. It is shown that we encounter the same behavior as for the small genus, except that we are unable to say much about hoppers and cascades.

\section{Preliminaries}
\label{sc-preliminaries}

Let $G$ be  a connected multigraph. An \df{embedding} of $G$ is a pair $\Pi = (\pi, \lambda)$
where $\pi = (\pi_v \mid v \in V(G))$ is a \df{rotation system}, which assigns each vertex
$v$ a cyclic permutation $\pi_v$ of the edges incident with $v$, and $\lambda$ is a \df{signature}
mapping which assigns each edge $e \in E(G)$ a \df{sign} $\lambda(e) \in \{-1, 1\}$.
For an edge $e$ incident to $v$, the cyclic sequence $e, \pi_v(e), \pi_v^2(e), \ldots, e$ is called
the \df{local rotation} at $v$.
Given an embedding $\Pi$ of $G$, we say that $G$ is \df{$\Pi$-embedded}.

A \df{$\Pi$-face} of a $\Pi$-embedded graph $G$ is a cyclic sequence of triples $(v_i, e_i, s_i)$,
where $v_i \in V(G)$, $e_i$ is an edge incident with $v_i$, and $s_i\in \{1,-1\}$,
satisfying the following (with indices taken modulo $k$, where $k$ is the length of the sequence):
\begin{enumerate}[label=(\roman*)]
\item
  $e_i = v_iv_{i+1}$,
\item
  $s_{i+1} = s_i\lambda(e_i)$, and
\item
  $e_{i+1} = \pi_{v_{i+1}}^{s_{i+1}}(e_i)$.
\end{enumerate}
Two consecutive tuples $(v, e, s), (v', e', s')$ of a $\Pi$-face $W$ give a \df{$\Pi$-angle} $e, v', e'$ of $W$.
Let $F(G,\Pi)$ be the set of $\Pi$-faces.
The \df{Euler genus} of $\Pi$ is given by Euler's formula.
$$\eg(\Pi) = 2 - |V(G)| + |E(G)| - |F(G,\Pi)|.$$
The \df{Euler genus} $\eg(G)$ of a graph $G$ is the minimum Euler genus among all embeddings of $G$.

If $G$ contains a cycle that has odd number of edges of negative signature,
we say that $\Pi$ is \df{nonorientable}. Otherwise, $\Pi$ is \df{orientable}.
The \df{orientable genus} $g(G)$ is half of the minimum genus of an orientable embedding of $G$.
If $G$ contains at least one cycle, then the \df{nonorientable genus} $\ng(G)$ is the minimum
Euler genus of a nonorientable embedding of $G$, else $\ng(G) = 0$.
The following relation is an easy observation (see~\cite{mohar-book}).

\begin{lemma}
\label{lm-ng-upper}
  For every connected graph $G$ which is not a tree, $\ng(G) \le 2g(G) + 1$.
\end{lemma}

If $\ng(G) = 2g(G) + 1$, then $G$ is said to be \df{orientably simple}.
Note that in this case $\eg(G) = 2g(G) = \ng(G)-1$, i.e., the Euler genus of $G$ is even.

In this paper, we will deal mainly with the class $\G$ of simple graphs.
Let $G \in \G$ be a simple graph and $e\in E(G)$.
Then $G-e$ denotes the graph obtained from $G$ by \df{deleting} $e$
and $G/e$ denotes the graph\footnote{When contracting an edge, one may obtain multiple edges.
We shall replace any multiple edges by single edges as such a simplification has no effect on the genus.}
obtained from $G$ by \df{contracting} $e$. It is convenient for us to formalize these graph operations.
The set $\M(G) = E(G) \times \{-,/\}$ is the \df{set of minor-operations} available for $G$.
An element $\mu \in \M(G)$ is called a \df{minor-operation} and $\mu G$ denotes the graph obtained
from $G$ by applying $\mu$. For example, if $\mu = (e, -)$ then $\mu G  = G - e$.
A graph $H$ is a \df{minor} of $G$ if $H$ can be obtained from a subgraph of $G$ by contracting some edges.
If $G$ is connected, then $H$ can be obtained from $G$ by a sequence of minor-operations.

We shall use the following well-known result.

\begin{theorem}[Stahl and Beineke~\cite{stahl-1977}]
\label{th-stahl-euler}
The Euler genus of a graph is the sum of the Euler genera of its blocks.
\end{theorem}

Generally, we are interested in minor-minimal graphs (with some property). The closely related classes
of deletion-minimal graphs appear naturally. For a surface $\SS$, let $\Forb^*(\SS)$ be the class of graphs of minimum degree at least 3 that do not embed into $\SS$ but are minimal such with respect to taking subgraphs. Similarly, let $\E^*_k$ be the class of graphs of minimum degree at least 3 such that $\eg(G) > k$ but $\eg(G - e) \le k$ for each edge $e \in E(G)$. Again, we let $\E^* = \bigcup_{k\ge 0} \E^*_k$.

\section{Graphs with terminals}
\label{sc-terminals}

We study the class $\Gxy$ of graphs with two special vertices $x$ and $y$, called \df{terminals}.
Most notions that are used for graphs can be used in the same way for graphs with terminals.
Some notions differ though and, to distinguish between graphs with and without terminals, let $\hat{G}$ be the underlying graph of $G$ without terminals (for $G \in \Gxy$).
Two graphs, $G_1$ and $G_2$, in $\Gxy$ are \df{isomorphic}, also denoted $G_1 \iso G_2$, if there is an isomorphism of the graphs $\hat{G_1}$ and $\hat{G_2}$
that maps terminals of $G_1$ onto terminals of $G_2$ (and non-terminals onto non-terminals), possibly exchanging $x$ and $y$.
We define minor-operations on graphs in $\Gxy$ in the way that $\Gxy$ is a minor-closed class.
When performing edge contractions on $G \in \Gxy$, we do not allow contraction of the edge $xy$ (if $xy \in E(G)$)
and when contracting an edge incident with a terminal, the resulting vertex becomes a terminal.

We use $\M(G)$ to denote the set of available minor-operations for $G$.
Since $(xy, /) \not\in \M(G)$ for $G \in \Gxy$,
we shall use $G /xy$ to denote the underlying simple graph in $\G$ obtained from $G$
by identification of $x$ and $y$. In particular, we do not require the edge $xy$ to be present in~$G$.

A \df{graph parameter} is a function $\G \into \RR$ that is constant on each isomorphism class of $\G$.
Similarly, we call a function $\Gxy \into \RR$ a \df{graph parameter} if it is constant on each isomorphism class of $\Gxy$.
A graph parameter $\P$ is \df{minor-monotone} if $\P(H) \le \P(G)$ for each graph $G \in \Gxy$ and each minor $H$ of $G$.
The Euler genus is an example of a minor-monotone graph parameter.

For $G \in \Gxy$, the graph $G\+$ is the graph obtained from $G$ by adding the edge $xy$ if it is not already present.
We can view the Euler genus of $G\+$ as a graph parameter $\egp$ of $G$, $\egp(G) = \eg(G\+)$.
Note that $\egp$ is minor-monotone.
The difference of $\egp$ and $\eg$ is a parameter $\et$, that is, $\et(G) = \eg(G\+) - \eg(G)$.
Note that $\et(G) \in \{0,1,2\}$.

\begin{table}
  \centering
  \begin{tabular}{| c | l | l |}
    \hline
    Parameter & Definition & Range \\
    \hline
    $\eg(G)$ & Euler genus & $\ge 0$ \\
    $\egp(G)$ & $\eg(G+xy)$ & $\ge 0$ \\
    $\et(G)$ & $\egp(G)-\eg(G)$ & $0,1,2$ \\
    $\eeta(G_1, G_2)$ & $\et(G_1) + \et(G_2)$ & $0,1,2,3,4$ \\
    \hline
  \end{tabular}
  \caption{Genus parameters for graphs in $\Gxy$}
  \label{tb-ranges}
\end{table}

Let $\P$ be a graph parameter.
A graph $G$ is \df{$\P$-critical} if $\P(\mu G) < \P(G)$ for each $\mu \in \M(G)$.
Let $H$ be a subgraph of a graph $G$ (possibly with terminals) and $\P$ a graph parameter. We say that $H$ is \df{$\P$-tight} if
$\P(\mu G) < \P(G)$ for every minor-operation $\mu \in \M(H)$.
We observe that $\P$-critical graphs have $\P$-tight subgraphs:

\begin{lemma}
\label{lm-tight}
Let $H_1, \ldots, H_s$ be subgraphs of a graph $G$ (possibly with terminals).
If $E(H_1) \cup \cdots \cup E(H_s) = E(G)$, then
$G$ is $\P$-critical if and only if $H_1, \ldots, H_s$ are $\P$-tight.
\end{lemma}

Let $\Gcxy$ be the subclass of $\Gxy$ that consists of graphs that do not contain the edge $xy$.
For graphs $G_1, G_2 \in \Gxy$ such that $V(G_1) \cap V(G_2) = \{x,y\}$, the graph $G = (V(G_1) \cup V(G_2), E(G_1) \cup E(G_2))$ is the \df{$xy$-sum} of $G_1$ and $G_2$.
The graphs $G_1$ and $G_2$ are called \df{parts} of $G$.
Let $G$ be the $xy$-sum of $G_1, G_2 \in \Gxy$.
We define the following two parameters:
\begin{equation}
  \label{eq-eg-0}
  \eh_0(G) = \eg(G_1) + \eg(G_2) + 2;
\end{equation}
\begin{equation}
  \label{eq-eg-1}
  \eh_1(G) = \egp(G_1) + \egp(G_2).
\end{equation}

Eq.~(\ref{eq-eg-1}) can be rewritten in a form similar to Eq.~(\ref{eq-eg-0}).
\begin{equation}
  \label{eq-thetas}
  \eh_1(G) = \eg(G_1) + \eg(G_2) + \et(G_1) + \et(G_2) = \eg(G_1) + \eg(G_2) + \eeta(G_1, G_2)
\end{equation}
where $\eeta(G_1, G_2) = \et(G_1) + \et(G_2)$.
Note that $\eeta(G_1, G_2) \in \{0,1,2,3,4\}$.

Richter~\cite{richter-1987-euler} gave a precise formula for the Euler genus of a 2-sum that
can be expressed using our notation as follows.

\begin{theorem}[Richter~\cite{richter-1987-euler}]
  \label{th-richter-euler}
  Let $G$ be the $xy$-sum of connected graphs $G_1, G_2 \in \Gxy$. Then
  \begin{enumerate}[label=\rm(\roman*)]
  \item
    $\eg(G) = \min \{\eh_0(G), \eh_1(G)\}$,
  \item
    $\egp(G) = \eh_1(G)$, and
  \item
    $\et(G) = \max\{\eh_1(G) - \eh_0(G), 0\}.$
  \end{enumerate}
\end{theorem}

We can rewrite (i) as
\begin{equation}
  \label{eq-eta-eg}
  \eg(G) = \eg(G_1) + \eg(G_2) + \min\{\eeta(G_1, G_2), 2\}
\end{equation}
and as
\begin{equation}
  \label{eq-eta-egp}
  \eg(G) = \egp(G_1) + \egp(G_2) + 2 - \max\{\eeta(G_1, G_2), 2\}.
\end{equation}

For a graph parameter $\P$, we say that a minor-operation $\mu \in \M(G)$ \df{decreases} $\P$ by at least $k$
if $\P(\mu G) \le \P(G) - k$. The subset of $\M(G)$ that decreases $\P$ by at least $k$ is denoted by $\dc k {\P,G}$.
We write just $\dc k \P$  when the graph is clear from the context.
Note that a graph $G$ is $\P$-critical precisely when $\M(G) = \dc1 \P$.
The following observation is stated for later reference.

\begin{lemma}
\label{lm-two-separated}
  Let $G \in \Gxy$. Then $\eg(G) \le \egp(G) \le \eg(G) + 2$. Furthermore, for $\et = \et(G)$ and $k \ge 0$, we have
  \begin{enumerate}
  \setlength{\itemindent}{3mm}
  \item[\rm(S1)]
    $\dc {k+2-\et} \eg \ss \dc k \egp$ and
  \item[\rm(S2)]
    $\dc {k+\et} \egp \ss \dc k \eg$.
  \end{enumerate}
\end{lemma}

\begin{proof}
Suppose that $\mu \in \dc {k+2-\et} \eg$, i.e., $\eg(G)\ge \eg(\mu G)+k+2-\et$.
Then $\egp(G) = \eg(G)+\et \ge \eg(\mu G)+k+2 \ge \egp(\mu G) + k$. This shows that $\mu\in \dc {k} \egp$ and proves (S1). Property (S2) is verified in the same way.
\end{proof}

As an example, take a graph $G$ with $\eg(G) = 1$ and $\egp(G) = 2$. Then (S2) for $k = 1$ says that
$\dc2 \egp \ss \dc1 \eg$, or that each minor-operation that decreases the Euler genus of $G\+$ by at least 2 also decreases
the Euler genus of $G$ by at least $1$.

The next lemma describes when a minor-operation in a part of a 2-sum decreases $\eg$ of the 2-sum.

\begin{lemma}
  \label{lm-part-eg}
  Let $G$ be the $xy$-sum of connected graphs $G_1, G_2 \in \Gcxy$ and let $\mu \in \M(G_1)$ be a minor-operation such that $\mu G_1$ is connected.
  Then $\eg(\mu G) < \eg(G)$ if and only if the following is true (where $\dc k\cdot$ always refer to the decrease of the parameter in $G_1$):
  \begin{enumerate}[label=\rm(\roman*)]
  \item
    If $\eeta(G_1, G_2) = 0$, then $\mu \in \dc1 \egp$.
  \item
    If $\eeta(G_1, G_2) = 1$, then $\mu \in \dc1 \egp \cup \dc2 \eg$.
  \item
    If $\eeta(G_1, G_2) = 2$, then $\mu \in \dc1 \egp \cup \dc1 \eg$.
  \item
    If $\eeta(G_1, G_2) = 3$, then $\mu \in \dc2 \egp \cup \dc1 \eg$.
  \item
    If $\eeta(G_1, G_2) = 4$, then $\mu \in \dc1 \eg$.
  \end{enumerate}
\end{lemma}

\begin{proof}
  Assume first that $\eg(\mu G) < \eg(G)$.
  Suppose that $\eeta(G_1, G_2) \le 2$ and that $\mu \not\in \dc1 \egp$.
  Since $\mu G_1$ is connected and $\eh_1(\mu G) = \eh_1(G)$, Theorem~\ref{th-richter-euler} gives that $\eg(\mu G) = \eh_0(G)$.
  Thus using Eq.~(\ref{eq-eta-eg}), we obtain that
  $$\eg(\mu G_1) + \eg(G_2) + 2 = \eh_0(\mu G) = \eg(\mu G) < \eg(G) \le \eg(G_1) + \eg(G_2) + \eeta(G_1, G_2).$$
  If $\eeta(G_1, G_2) = 0$, then $\eg(\mu G_1) < \eg(G_1) - 2$. Thus $\mu \in \dc3 \eg$ and, since  $\dc3 \eg \ss \dc1 \egp$ by (S1), $\mu \in \dc1 \egp$, a contradiction.
  We conclude that (i) holds.
  If $\eeta(G_1, G_2) = 1$, then $\mu \in \dc2 \eg$ and (ii) holds.
  If $\eeta(G_1, G_2) = 2$, then $\mu \in \dc1 \eg$ and (iii) holds.

  Assume now that $\eeta(G_1, G_2) \ge 3$ and assume that $\mu \not\in \dc1 \eg$. Then $\eh_0(\mu G) = \eh_0(G)$.
  Consequently, by Theorem~\ref{th-richter-euler}, $\eg(\mu G) = \eh_1(G)$. Thus using Eq.~(\ref{eq-eta-egp}), we have
  $$\egp(\mu G_1) + \egp(G_2) = \eh_1(\mu G) = \eg(\mu G) < \eg(G) \le \egp(G_1) + \egp(G_2) + 2 - \eeta(G_1, G_2).$$
  If $\eeta(G_1, G_2) = 3$, then $\egp(\mu G_1) < \egp(G_1) - 1$. Hence $\mu \in \dc2 \egp$ and (iv) holds.
  If $\eeta(G_1, G_2) = 4$, then $\egp(\mu G_1) < \egp(G_1) - 2$. Thus $\mu \in \dc3 \egp$ and, since $\dc3 \egp \ss \dc1 \eg$ by (S2), $\mu \in \dc1 \eg$, a contradiction.
  We conclude that (v) holds.

  To prove the ``if'' part of the lemma, assume that (i)--(v) hold.
  We need to show that $\eg(\mu G) < \eg(G)$.
  Assume first that $\eeta(G_1, G_2) \le 2$ and $\mu \in \dc1 \egp$. By Theorem~\ref{th-richter-euler} and~(\ref{eq-eta-egp}),
  $$\eg(\mu G) \le \eh_0(\mu G) = \egp(\mu G_1) + \egp(G_2) < \egp(G_1) + \egp(G_2) = \eh_0(G) = \eg(G).$$

  Assume now that $\eeta(G_1, G_2) \ge 2$ and $\mu \in \dc1 \eg$. By Theorem~\ref{th-richter-euler} and~(\ref{eq-eta-eg}),
  $$\eg(\mu G) \le \eh_1(\mu G) = \eg(\mu G_1) + \eg(G_2) + 2 < \eg(G_1) + \eg(G_2) + 2 = \eh_1(G) = \eg(G).$$

  If $\eeta(G_1, G_2) = 1$ and $\mu \in \dc2 \eg$, then using~(\ref{eq-eta-eg}),
  $$\eg(\mu G) \le \eh_1(\mu G) = \eg(\mu G_1) + \eg(G_2) + 2 < \eg(G_1) + \eg(G_2) + 1 = \eg(G).$$

  If $\eeta(G_1, G_2) = 3$ and $\mu \in \dc2 \egp$, then using~(\ref{eq-eta-egp}),
  $$\eg(\mu G) \le \eh_0(\mu G) = \egp(\mu G_1) + \egp(G_2) < \egp(G_1) + \egp(G_2) -1 = \eg(G).$$
  Since the cases (i)--(v) cover all possible values of $\eeta$, at least one of the cases above occurs and we are done.
\end{proof}

Let us prove a similar lemma for $\egp$.

\begin{lemma}
  \label{lm-part-egp}
  Let $G$ be the $xy$-sum of connected graphs $G_1, G_2 \in \Gcxy$ and let $\mu \in \M(G_1)$ be a minor-operation such that $\mu G_1$ is connected.
  Then $\egp(\mu G) < \egp(G)$ if and only if $\mu \in \dc1 \egp$.
\end{lemma}

\begin{proof}
  Assume first that $\egp(\mu G) < \egp(G)$.
  Since $\mu G_1$ is connected, Theorem~\ref{th-richter-euler} gives that $\egp(\mu G) = \eh_1(G)$. Using Eq.~(\ref{eq-eg-1}), we have that
  $$\egp(\mu G_1) + \egp(G_2) = \eh_1(\mu G) = \egp(\mu G) < \egp(G) = \eh_1(G) = \egp(G_1) + \eg(G_2).$$
  Thus $\egp(\mu G_1) < \egp(G_1)$. Hence $\mu \in \dc1 \egp$.

  On the other hand, assume that $\mu \in \dc1 \egp$.
  Thus $\egp(\mu G_1) < \egp(G_1)$.
  By Theorem~\ref{th-richter-euler} and Eq.~(\ref{eq-eg-1}),
  $$\egp(\mu G) = \egp(\mu G_1) + \egp(G_2) < \egp(G_1) + \egp(G_2) = \egp(G),$$
  as claimed.
\end{proof}

In the statements of Lemmas~\ref{lm-part-eg} and~\ref{lm-part-egp}, we required that $\mu G_1$ is connected.
The next lemma shows that this is indeed the case for all minor-operations if $G_1$ is $\eg$-tight or $\egp$-tight in $G$.
It is not hard to see that if $G_1$ is a connected graph, then $\mu G_1$ is disconnected if and only if $\mu$ is the deletion of a cutedge of $G_1$.

\begin{lemma}
  \label{lm-no-cutedges}
  Let $G \in \Gcxy$ be a connected graph with a cutedge $e$.
  Then $\eg(G/e) = \eg(G)$ and $\egp(G/e) = \egp(G)$.
\end{lemma}

\begin{proof}
  Let $H_1$ and $H_2$ be the components of $G - e$.
  By Theorem~\ref{th-stahl-euler}, $\eg(G/e) = \eg(H_1) + \eg(H_2) = \eg(G)$.
  If both $x$ and $y$ lie in $H_1$ (or $H_2$ by symmetry), then
  by Theorem~\ref{th-stahl-euler}, $\egp(G/e) = \eg(H_1 + xy) + \eg(H_2) = \egp(G)$.
  Suppose then that $x \in V(H_1)$ and $y \in V(H_2)$.
  If $x$ (or $y$ by symmetry) and $w$ are the endpoints of $e$, then $G\+$ is the $1$-sum of $H_1$ and $H_2 + xy + xw$.
  Since $\eg(H_2 + yw) = \eg(H_2 + xy + xw)$, we have that $\egp(G/e) = \eg(H_1) + \eg(H_2 + yw) = \eg(H_1) + \eg(H_2 + xy + e) = \egp(G)$.

  Therefore we may assume that $e$ has endpoints $z \in V(H_1) \sm \{x\}$ and $w \in V(H_2) \sm \{y\}$.
  Let us view the graph $G\+$ as a $yz$-sum of graphs $H_1' = H_1+xy$ and $H_2' = H_2 + e$.
  We have that $(e,/) \in \M(H_2')$ and $\eg(H_2'/e) = \eg(H_2')$ by Theorem~\ref{th-stahl-euler} since $e$ is a block of $H_2'$.
  Similarly, $\egp(H_2'/e) = \egp(H_2')$ since $H_2'/e$ is homeomorphic to $H_2'$ and thus admits the same embeddings.
  By applying Theorem~\ref{th-stahl-euler} to $G\+$ as a $yz$-sum of $H_1'$ and $H_2'$, we obtain that $\eg(G\+/e) = \eg(G\+)$.
  We conclude that $\egp(G/e) = \egp(G)$.
\end{proof}

\begin{table}
  \centering
  \begin{tabular}{|c | c |}
    \hline
    $\eeta(G_1, G_2)$ & $\M(G_1)$ \\
    \hline
    0 & $\dc1 \egp$ \\
    1 & $\dc1 \egp \cup \dc2 \eg$ \\
    2 & $\dc1 \egp \cup \dc1 \eg$ \\
    3 & $\dc2 \egp \cup \dc1 \eg$ \\
    4 & $\dc1 \eg$ \\
    \hline
  \end{tabular}
  \caption{Possible outcomes for a minor-operation in a $\eg$-tight part of a 2-sum.}
  \label{tb-parts-eg}
\end{table}

Lemma~\ref{lm-no-cutedges} easily implies that a $\eg$-tight or $\egp$-tight part $G_1$ of an $xy$-sum $G$
has no cutedges. If $e$ is a cutedge of $G_1$, then $G_1 / e$ is connected, $\eg(G_1 /e) = \eg(G_1)$, and $\egp(G_1/e) = \egp(G_1)$.
By Lemmas~\ref{lm-part-eg} and~\ref{lm-part-egp}, $G_1$ is neither $\eg$-tight nor $\egp$-tight in $G$.
In particular, we may present the outcome of Lemma~\ref{lm-part-eg} in terms of $\M(G_1)$ as in Table~\ref{tb-parts-eg}.

\section{Critical classes, cascades, and hoppers}
\label{sc-klein-critical}

For a graph parameter $\P$, let $\C(\P)$ denote the class of $\P$-critical graphs in $\Gxy$.
Note that $G \in \C(\P)$ if and only if $\M(G) = \dc1 \P$.
We call $\C(\P)$ the \df{critical class} for $\P$.
Let $\Cc(\P)$ be the class $\C(\P) \cap \Gcxy$.
We refine the class $\C(\P)$ according to the value of $\P$:
Let $\C_k(\P)$ denote the subclass of $\C(\P)$ that contains precisely the graphs $G$ for which $\P(G) = k+1$. Let $\Cc_k(\P)$ be the class $\C_k(\P) \cap \Gcxy$ of those $\P$-critical graphs that do not contain the edge $xy$.

Let us start this section by describing the relation between the classes $\Cc(\eg)$, $\Cc(\egp)$, and $\E$ (unlabeled graphs that are critical for the Euler genus).
The next result follows from the definitions of $\E$ and $\Cc(\eg)$.

\begin{lemma}
\label{lm-cc-eg}
  For $G \in \Gcxy$, $\hat{G} \in \E$ if and only if $G \in \Cc(\eg)$.
\end{lemma}

The next two lemmas describe the relation between the class $\Cc(\egp)$ and $\E$.

\begin{lemma}
\label{lm-cc-egp-iff}
  For $G \in \Gcxy$, $\hat{G\+} \in \E$ if and only if $G \in \Cc(\egp)$, $\et(G) > 0$, and $\eg(G / xy) < \egp(G)$.
\end{lemma}

\begin{proof}
  Let $H = \hat{G\+}$.
  Note that $\eg(H) = \egp(G)$ and $\M(H) = \M(G) \cup \{(xy, -), (xy, /)\}$.
  Since $\eg(\mu H) = \egp(\mu G)$ for each $\mu \in \M(G)$,
  we get that $\eg(\mu H) < \eg(H)$ for each $\mu \in \M(G)$ if and only if $G \in \Cc(\egp)$.
  Since $H - xy \iso \hat{G}$, we obtain that $\eg(H - xy) < \eg(H)$ if and only if $\et(G) > 0$.
  Since $H/xy \iso G/xy$, we have that $\eg(H / xy) < \eg(H)$ if and only if $\eg(G / xy) < \egp(G)$.
  As $H \in \E$ if and only if $\eg(\mu H) < \eg(H)$ for each $\mu \in \M(H)$, the result follows.
\end{proof}

\begin{lemma}
\label{lm-cc-egp}
Let $G \in \Cc(\egp)$. If $\et(G) = 0$, then $\hat{G} \in \E$.
If $\et(G) > 0$, then either $\hat{G\+} \in \E$, or $\hat{G\+} \in \E^*$ and $\hat{G / xy} \in \E$.
\end{lemma}

\begin{proof}
  If $\et(G) = 0$, then $\M(G) = \dc1 \eg$ by (S2) and thus $G \in \Cc(\eg)$.
  Therefore $\hat{G} \in \E$ by Lemma~\ref{lm-cc-eg}.
  Suppose now that $\et(G) > 0$. Let $H = \hat{G\+}$.
  Since $G \in \Cc(\egp)$, we have that $\eg(\mu H) < \eg(H)$ for each $\mu \in \M(G)$.
  As $\eg(H - xy) = \eg(G) < \eg(G) + \et(G) = \eg(H)$,
  we have that $H \in \E^*$.
  If $\eg(G / xy) < \egp(G)$, then $H \in \E$ (since both deletion and contraction of $xy$ decrease the Euler genus of $H$).
  Hence we may assume that $\eg(G /xy) = \egp(G)$.
  Let $\mu \in \M(G /xy)$ be a minor-operation in $G / xy$.
  Since $\mu$ is also a minor-operation in $G$, we obtain that
  $$\eg(\mu(G /xy)) \le \eg(\mu G\+) = \egp(\mu G) < \egp(G) = \eg(G / xy)$$
  as $\mu(G /xy)$ is a minor of $\hat{\mu G\+}$.
  Since $\mu$ was chosen arbitrarily, $G/xy \in \E$.
\end{proof}

A graph $G \in \Gcxy$ is called a \df{cascade} if $G$ satisfies the following properties:
\begin{enumerate}[label=(C\arabic*)]
\item
  $\M(G) = \dc1 \eg \cup \dc1 \egp$ (i.e., each minor operation decreases $\eg$ or $\egp$).
\item
  $G \not\in \Cc(\eg)$ (i.e., some minor operation does not decrease $\eg$).
\item
  $G \not\in \Cc(\egp)$ (i.e., some minor operation does not decrease $\egp$).
\end{enumerate}

Let $\S$ be the class of all cascades.
We refine the class $\S$ according to the Euler genus.
Let $\S_k$ be the subclass of $\S$ containing those graphs $G$ for which $\egp(G) = k+1$.
It is not hard to see that for $G \in \S_k$ we have that $\eg(G) = k$.

\begin{lemma}
  \label{lm-cascades}
  If $G \in \S$, then $\et(G) = 1$.
\end{lemma}

\begin{proof}
  If $\et(G) = 0$, then $\dc1 \egp \ss \dc1 \eg$ by (S2), violating (C2).
  If $\et(G) = 2$, then $\dc1 \eg \ss \dc1 \egp$ by (S1), violating (C3).
  Thus $\et(G) = 1$.
\end{proof}

In this paper we shall show that the class of cascades is nonempty.
In particular, we will determine the class $\S_1$ which appears as a class of building blocks for obstructions of connectivity 2 for the Klein bottle.
The following lemma is an immediate consequence of (C1)--(C3).

\begin{lemma}
\label{lm-eg-and-egp}
  Let $G \in \Gcxy$. If $\M(G) = \dc1 \eg \cup \dc1 \egp$, then
  $G \in \Cc(\eg) \cup \Cc(\egp) \cup \S$.
\end{lemma}


We shall encounter another class of building blocks for obstructions of connectivity two. This class is more mysterious and we call them hoppers. Although it turns out that they do not exist when the genus is small (see Lemma~\ref{lm-hoppers}), we suspect that they might appear when the genus becomes large. Their existence or nonexistence is intimately related to an old open question if there exist graphs that are obstructions for two different nonorientable surfaces.

Let $G \in \Gcxy$.
For a graph parameter $\P$, a graph $G$ is a \df{$\P$-hopper} if every minor operation reduces the parameter by at least 2, i.e., $\M(G) = \dc2 \P$.
Let $\H(\P)$ be the class of $\P$-hoppers.
The subclass of $\H(\P)$ of graphs with $\P$ equal to $k+1$ is denoted by $\H_k(\P)$.
In this paper, we restrict our attention to $\eg$-hoppers and $\egp$-hoppers.

Let us define two weaker forms of hoppers.
We say that $G$ is a \df{weak $\eg$-hopper} if $G \not\in \Cc(\egp)$ and $\M(G) = \dc1 \egp \cup \dc2 \eg$.
Note that necessarily $\et(G) = 0$ by (S1); and $G \in \Cc(\eg)$ by (S2).
We say that $G$ is a \df{weak $\egp$-hopper} if $G \not\in \Cc(\eg)$ and $\M(G) = \dc1 \eg \cup \dc2 \egp$.
Note that $\et(G) = 2$ by (S2) and $G \in \Cc(\egp)$ by (S1).
Let $\Hw(\eg)$ and $\Hw(\egp)$ be the class of weak $\eg$-hoppers and weak $\egp$-hoppers, respectively.
Let $\Hw_k(\P)$ be the subclass of $\Hw(\P)$ such that $G \in \Hw_k(\P)$ if $\P(G) = k+1$.
The next result follows directly from the definition of weak hoppers.

\begin{lemma}
\label{lm-weak-hoppers}
  Let $G \in \Gcxy$.
  If $\M(G) = \dc2 \eg \cup \dc1 \egp$, then $G \in \Cc(\egp) \cup \Hw(\eg)$.
  If $\M(G) = \dc1 \eg \cup \dc2 \egp$, then $G \in \Cc(\eg) \cup \Hw(\egp)$.
\end{lemma}

For the record we also state the following observation.

\begin{observation}
  The class $\H_k(\eg)$ is empty if and only if each graph $G \in \E_{k-1}$ has $\eg(G) = k$.
\end{observation}


Let us now combine the properties of introduced classes with Lemma~\ref{lm-part-eg} to
characterize $\eg$-tight and $\egp$-tight parts of a 2-sum of two graphs.

\begin{table}
  \centering
  \begin{tabular}{|c | c |}
    \hline
    $\eeta(G_1, G_2)$ & $G_1$ \\
    \hline
    0 & $\Cc(\egp)$ \\
    1 & $\Cc(\egp) \cup \Hw(\eg)$ \\
    2 & $\Cc(\egp) \cup \Cc(\eg) \cup \S$ \\
    3 & $\Cc(\eg) \cup \Hw(\egp)$ \\
    4 & $\Cc(\eg)$ \\
    \hline
  \end{tabular}
  \caption{Classification of $\eg$-tight parts of a 2-sum.}
  \label{tb-general-eg}
\end{table}

\begin{theorem}
  \label{th-general-eg}
  Let $G$ be the $xy$-sum of connected graphs $G_1, G_2 \in \Gcxy$.
  The subgraph $G_1$ is $\eg$-tight in $G$ if and only if the following is true:
  \begin{enumerate}[label=\rm(\roman*)]
  \item
    If $\eeta(G_1, G_2) = 0$, then $G_1 \in \Cc(\egp)$.
  \item
    If $\eeta(G_1, G_2) = 1$, then $G_1 \in \Cc(\egp) \cup \Hw(\eg)$.
  \item
    If $\eeta(G_1, G_2) = 2$, then $G_1 \in \Cc(\egp) \cup \Cc(\eg) \cup \S$.
  \item
    If $\eeta(G_1, G_2) = 3$, then $G_1 \in \Cc(\eg) \cup \Hw(\egp)$.
  \item
    If $\eeta(G_1, G_2) = 4$, then $G_1 \in \Cc(\eg)$.
  \end{enumerate}
\end{theorem}

\begin{proof}
  Assume first that $G_1$ is $\eg$-tight.
  By Lemma~\ref{lm-no-cutedges}, $\mu G_1$ is connected for each $\mu \in \M(G_1)$.
  If $\eeta(G_1, G_2) = 0$, then $\M(G_1) = \dc1 \egp$ by Lemma~\ref{lm-part-eg}.
  Thus $G_1 \in \Cc(\egp)$. Similarly, if $\eeta(G_1, G_2) = 4$, then $G_1 \in \Cc(\eg)$.
  If $\eeta(G_1, G_2) = 1$, then $\M(G_1) = \dc1 \egp \cup \dc2 \eg$. By Lemma~\ref{lm-weak-hoppers}, $G_1 \in \Cc(\egp) \cup \Hw(\eg)$.
  If $\eeta(G_1, G_2) = 3$, then $\M(G_1) = \dc1 \eg \cup \dc2 \egp$. By Lemma~\ref{lm-weak-hoppers}, $G_1 \in \Cc(\eg) \cup \Hw(\egp)$.
  Finally, if $\eeta(G_1, G_2) = 2$, then $\M(G_1) = \dc1 \eg \cup \dc1 \egp$. By Lemma~\ref{lm-cascades}, $G_1 \in \Cc(\eg) \cup \Cc(\egp) \cup \S$.

  Assume now that (i)--(v) hold.
  Since $\M(G) = \dc1 \eg \cup \dc1 \egp$ for $G \in \Cc(\eg) \cup \Cc(\egp) \cup \S \cup \Hw(\eg) \cup \Hw(\egp)$,
  Lemma~\ref{lm-no-cutedges} asserts that $\mu G_1$ is connected for each $\mu \in \M(G_1)$.
  Suppose first that $G_1 \in \Cc(\eg)$.
  Since $\M(G) = \dc1 \eg$, we obtain for each $\eeta(G_1, G_2) \in \{2,3,4\}$ that $G_1$ is $\eg$-tight by Lemma~\ref{lm-part-eg}.
  A similar argument works if $G_1 \in \Cc(\egp)$ and $\eeta(G_1, G_2) \in \{0, 1, 2\}$.
  If $\eeta(G_1, G_2) = 1$ and $G_1 \in \Hw(\eg)$, then  $\M(G_1) = \dc1 \egp \cup \dc2 \eg$ and $G_1$ is $\eg$-tight by Lemma~\ref{lm-part-eg}.
  If $\eeta(G_1, G_2) = 3$ and $G_1 \in \Hw(\egp)$, then $\M(G_1) = \dc2 \egp \cup \dc1 \eg$ and $G_1$ is $\eg$-tight by Lemma~\ref{lm-part-eg}.
  If $\eeta(G_1, G_2) = 2$ and $G_1 \in \S$, then $\M(G_1) = \dc1 \eg \cup \dc1 \egp$ by (C1) and $G_1$ is $\eg$-tight by Lemma~\ref{lm-part-eg}.
  This completes the proof since $\eeta(G_1, G_2) \in \{0,\ldots,4\}$ and we have proven that $G_1$ is $\eg$-tight in each case given by (i)--(v).
\end{proof}

The outcome of Theorem~\ref{th-general-eg} is summarized in Table~\ref{tb-general-eg}.
There is an analogous theorem for $\egp$-tight parts of 2-sums.

\begin{theorem}
  \label{th-general-egp}
  Let $G$ be the $xy$-sum of connected graphs $G_1, G_2 \in \Gcxy$.
  The subgraph $G_1$ is $\egp$-tight in $G$ if and only if $G_1 \in \Cc(\egp)$.
\end{theorem}

\begin{proof}
  By Lemmas~\ref{lm-part-egp} and~\ref{lm-no-cutedges}, $G_1$ is $\egp$-tight if and only if $\M(G_1) = \dc1 \egp$.
  By definition, $G_1 \in \Cc(\egp)$ if and only if $\M(G_1) = \dc1 \egp$.
\end{proof}

\section{Excluded minors for Euler genus 2}
\label{sc-klein-euler}

\begin{figure}
  \centering
  \includegraphics{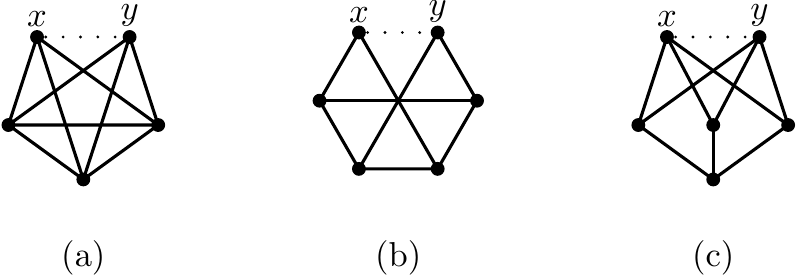}
  \caption{The class $\Cc_0(\egp)$, the third graph is the sole member of the class $\Cc_0(\eg)$.}
  \label{fg-cc-0-egp}
\end{figure}

In this section, we determine the classes $\Cc_2(\eg)$, $\Cc_2(\egp)$, and $\E_2$.
We begin by showing that the classes $\Cc_0(\eg)$ and $\Cc_0(\egp)$ are related to Kuratowski graphs $K_5$ and $K_{3,3}$.

\begin{lemma}
\label{lm-cc-0}
  The class $\Cc_0(\eg)$ consists of a single graph that is isomorphic to $K_{3,3}$ with non-adjacent terminals (Fig.~\ref{fg-cc-0-egp}c).
  The class $\Cc_0(\egp)$ consists of the three graphs shown in Fig.~\ref{fg-cc-0-egp}.
\end{lemma}

\begin{proof}
  A graph has Euler genus greater than 0 if and only if it is non-planar.
  Since both $K_5$ and $K_{3,3}$ embed into projective plane, $\E_0 = \Forb(\SS_0) = \{K_5, K_{3,3}\}$.
  By Lemma~\ref{lm-cc-eg},  a graph $G$ belongs to $\Cc_0(\eg)$ if only if $\hat{G} \in \E$.
  Since $xy \not\in E(G)$, $\hat{G}$ is not isomorphic to $K_5$ and
  thus $\Cc_0(\eg)$ consists of the unique graph isomorphic to $K_{3,3}$ with two non-adjacent terminals.

  Let us show first that each graph in Fig.~\ref{fg-cc-0-egp} belongs to $\Cc_0(\egp)$.
  If $\hat{G\+}$ is isomorphic to a Kuratowski graph, then $G \in \Cc_0(\egp)$ by Lemma~\ref{lm-cc-egp-iff}.
  Otherwise $\hat{G}$ is isomorphic to $K_{3,3}$ with $x$ and $y$ non-adjacent.
  It suffices to show that $\mu G\+$ is planar for each minor-operation $\mu \in \M(G)$ as $G\+$ clearly embeds into the projective plane.
  Pick an arbitrary edge $e \in E(G)$.
  The graph $G\+ - e$ has 9 edges and is not isomorphic to $K_{3,3}$ as it contains a triangle.
  The graph $G\+ / e$ has only 5 vertices and (at most) 9 edges.
  Since $e$ was arbitrary, it follows that $\mu G\+$ is planar for every $\mu \in \M(G)$.
  We conclude that $G \in \Cc_0(\egp)$.

  We shall show now that there are no other graphs in $\Cc_0(\egp)$.
  Let $G \in \Cc_0(\egp)$. By Lemma~\ref{lm-cc-egp}, there is a graph $H \in \Forb^*(\SS_0)$ such that either $\hat{G}$ is isomorphic to $H$
  or $G$ is isomorphic to the graph obtained from $H$ by deleting an edge and making the ends terminals.
  It is not hard to see that this yields precisely the graphs in Fig.~\ref{fg-cc-0-egp}.
\end{proof}

Note that the first two graphs in Fig.~\ref{fg-cc-0-egp} have $\et$ equal to 1 and the last one has $\et$ equal to~0.
We summarize the properties of graphs in $\Cc_0(\egp)$ in the following lemma.

\begin{lemma}
  \label{lm-cc-0-prop}
  For every $G \in \Cc_0(\egp)$, $G /xy$ is planar, $\et(G) \le 1$, and
  $\et(G) = 1$ if and only if $G \not\in \Cc_0(\eg)$.
\end{lemma}


Let us now consider the classes $\Cc_1(\eg)$ and $\Cc_1(\egp)$.
Since a graph embeds into the projective plane if and only if it has Euler genus at most 1,
we have that $\E_1 = \Forb(\NN_1)$.
Lemma~\ref{lm-cc-eg} says that $\Cc_1(\eg)$ can be constructed from the graphs $G$ in $\E_1$ with $\eg(G) = 2$ by choosing two nonadjacent vertices as terminals.
Actually, each graph $G \in \E_1$ has $\eg(G) = 2$.
This construction yields 195 (labeled) graphs in $\Cc_1(\eg)$ and confirms that the list is complete.
Note that while there are 35 graphs in $\E_1$, the class $\Cc_1(\eg)$ is larger because graphs in $\Gxy$
have two labeled terminals.

Lemma~\ref{lm-cc-egp} provides a mean for constructing the class $\Cc_1(\egp)$.
We construct a slightly larger class and then test which of the obtained graphs are in $\Cc_1(\egp)$.
Let $G \in \Cc_1(\egp)$. If $\et(G) = 0$, then $G \in \Cc_1(\eg)$ and thus $\hat{G} \in \E_1 \ss \E^*_1$.
If $\et(G) > 0$, then $\hat{G\+} \in \E^*_1$. The class $\E^*_1$ contains 103 graphs (see~\cite{archdeacon-1981}).
Let $\A$ be the class of graphs with terminals obtained from $\E^*_1$ by either making two nonadjacent
vertices terminals or deleting an edge $e$ and making the ends of $e$ terminals.
By Lemma \ref{lm-cc-egp}, we have that $\Cc_1(\egp)\subseteq \A$.
In order to construct $\Cc_1(\egp)$, it is sufficient to check which graphs $G$ in $\A$ are minor-minimal graphs
such that $G\+$ does not embed into the projective plane.
This construction gives 250 such graphs, out of which only 227 graphs have $G\+$ 2-connected.
The intersection $\Cc_1(\eg) \cap \Cc_1(\egp)$ contains 95 graphs, so we have 132 graphs in
$\Cc_1(\egp) \setminus \Cc_1(\eg)$.

The class $\S_1$ is determined in 
Part II \cite{MS2014_II}, where it is shown that $\S_1$ contains $21$ graphs (and that all of them have $G\+$ 2-connected). 

By considering all 348 graphs in $\Cc_1(\eg) \cup \Cc_1(\egp) \cup \S_1$,
we obtained the following result by using computer.

\begin{lemma}
  \label{lm-all-in-klein}
  For every $G \in \Cc_1(\eg) \cup \Cc_1(\egp) \cup \S_1$, the graph $G\+$ embeds into the Klein bottle.
\end{lemma}

To prove Lemma~\ref{lm-all-in-klein}, it is sufficient to provide an embedding of $\hat{G\+}$ in the Klein bottle for each $G \in \Cc_1(\eg) \cup \Cc_1(\egp) \cup \S_1$.
The graphs and their embeddings in the Klein bottle are available online\footnote{Embeddings of $G\+$ in the Klein bottle for every $G \in \Cc_1(\eg) \cup \Cc_1(\egp) \cup \S_1$ are available as the supplementary appendix in the arXiv version of this work.}.
Based on this evidence, we obtain the following properties of graphs in $\Cc_1(\eg)$.

\begin{lemma}
\label{lm-cc-1-eg}
  For every $G \in \Cc_1(\eg)$, we have that $\et(G) = 0$ and $\dc2 \eg \ss \dc1 \egp$.
\end{lemma}

\begin{proof}
  By Lemma~\ref{lm-all-in-klein}, $\egp(G) = \eg(G\+) \le 2$.
  Since $\eg(G) = 2$, we have that $\et(G) = \egp(G) - \eg(G) = 0$.

  The claim that $\dc2 \eg \ss \dc1 \egp$ was checked by computer.
  It is enough to show that for each $\mu \in \M(G)$ such that $\mu G$ is planar,
  the graph $\mu G\+$ is projective planar.
\end{proof}

The class of hoppers is mysterious. Although we were not able to construct any, we believe that they appear when the genus is large. However, there are none when genus is small.

\begin{lemma}
  \label{lm-hoppers}
  The classes $\Hw_1(\eg), \Hw_1(\egp)$, $\H_1(\eg)$, and $\H_1(\egp)$ are empty.
\end{lemma}

\begin{proof}
  Let $G \in \H_1(\eg)$.
  Since $G$ is non-planar, it has a Kuratowski graph $K$ as a minor. Since $\eg(G) = 2$, $K$ is a proper minor of $G$.
  Hence there is a minor-operation $\mu \in \M(G)$ such that $\mu G$ still has $K$ as a minor.
  Thus $\eg(\mu G) \ge \eg(K) = 1$. We conclude that $\mu \not\in \dc2 \eg$, a contradiction.

  Similarly, let $G \in \H_1(\egp)$. Then $\eg(G\+)=2$, and thus there is a Kuratowski graph $K$ that is a proper minor of $G\+$.
  Thus there is a minor-operation $\mu \in \M(\hat{G\+})$ such that $\hat{\mu G\+}$ has $K$ as a minor.
  Furthermore, since $\eg(K + uv) = 1$ for all $u,v \in V(G)$ by Lemma~\ref{lm-cc-0}, we may
  pick $\mu$ that does not delete nor contract $xy$. Thus $\mu \in \M(G)$.
  We have that $\egp(\mu G) \ge \eg(K) = 1$, a contradiction.

  Let $G \in \Hw_1(\egp)$. Thus $\egp(G) = 2$ and $\et(G) = 2$. Since $\eg(G) = 0$, we have that $\dc1 \eg = \emptyset$.
  We conclude that $\M(G) = \dc2 \egp$.
  Hence $G \in \H_1(\egp)$ which was already shown to be empty.

  Let $G \in \Hw_1(\eg)$. Thus $\egp(G) = 2$, $\et(G) = 0$, and $G \in \Cc_1(\eg)$.
  By Lemma~\ref{lm-cc-1-eg}, $\M(G) = \dc1 \egp$. Thus $G \in \Cc(\egp)$, a contradiction.
\end{proof}

\begin{table}
  \centering
  \begin{tabular}{|c | c |}
    \hline
    $\eeta(G_1, G_2)$ & $G_1$ \\
    \hline
    0 & $\Cc(\egp)$ \\
    1 & $\Cc(\egp)$ \\
    2 & $\Cc(\egp) \cup \Cc(\eg) \cup \S$ \\
    \hline
  \end{tabular}
  \caption{Classification of $\eg$-tight parts of a 2-sum in $\Cc_2(\eg)$.}
  \label{tb-klein}
\end{table}

Let us now state some properties of the parts of $xy$-sums in $\Cc_2(\eg)$ and $\Cc_2(\egp)$.

\begin{lemma}
  \label{lm-bounds-eg}
  Let $G$ be the $xy$-sum of connected graphs $G_1, G_2 \in \Gcxy$ such that~$\egp(G_1) \le \egp(G_2)$.
  If $G \in \Cc_2(\eg)$, then
  \begin{enumerate}[label=\rm(\roman*)]
  \item
    $\egp(G_1) = 1$,
  \item
    $\egp(G_2) = 2$,
  \item
    $\eeta(G_1, G_2) \le 2$.
  \end{enumerate}
\end{lemma}

\begin{proof}
  If $\egp(G_2) > 2$, then since $G_2^+$ is a proper minor of $G$, there is a minor-operation $\mu \in \M(G)$ such that $\eg(\mu G) \ge \eg(G_2^+) > 2$,
  a contradiction. Thus $\egp(G_2) \le 2$.
  If $\egp(G_1) = 0$, then $\eg(G) \le \eh_1(G) = \egp(G_1) + \egp(G_2) \le 2$ by Theorem~\ref{th-richter-euler}, a contradiction.
  Hence $\egp(G_1) \ge 1$.

  By Theorem~\ref{th-richter-euler}, we have
  $$\eh_1(G) = \egp(G_1) + \egp(G_2) \ge \eg(G) = 3.$$
  This implies that $\egp(G_2) = 2$ and (ii) holds.
  We also have
  $$\eh_0(G) = \eg(G_1) + \eg(G_2) + 2 \ge \eg(G) = 3.$$
  Therefore, $\eg(G_1) + \eg(G_2) \ge 1$.

  Suppose that $\egp(G_1) = 2$. If $\eg(G_1) + \eg(G_2) \ge 2$, then by Theorem~\ref{th-richter-euler},
  $$\eg(G) = \min\{\eh_0(G), \eh_1(G)\} = 4,$$
  a contradiction with $\eg(G) = 3$.
  Hence $\eg(G_1) + \eg(G_2) = 1$.
  Since $\egp(G_1) = \egp(G_2)$, we may exchange the roles of $G_1$ and $G_2$ if necessary and thus assume that $\eg(G_1) = 0$.
  By Lemma~\ref{lm-hoppers}, $\H_1(\egp) = \emptyset$ and thus there exists a minor-operation $\mu \in \M(G_1)$ such that $\egp(\mu G_1) \ge 1$.
  Note that $\eg(\mu G_1) = 0$.
  By Theorem~\ref{th-richter-euler},
  $$\eg(\mu G) = \min\{\eh_0(\mu G),\eh_1(\mu G)\} = \min\{\eg(\mu G_1) + \eg(G_2) + 2, \egp(\mu G_1) + \egp(G_2)\} = 3,$$
  a contradiction with $G \in \Cc_2(\eg)$.
  We conclude that $\egp(G_1) = 1$ and (i) holds.
  Since $\eg(G_1) + \eg(G_2) \ge 1$ and $\egp(G_1) + \egp(G_2) = 3$, we have that $\eeta(G_1, G_2) \le 2$ and (iii) holds.
\end{proof}

\begin{lemma}
  \label{lm-bounds-egp}
  Let $G$ be the $xy$-sum of connected graphs $G_1, G_2 \in \Gcxy$ such that~$\egp(G_1) \le \egp(G_2)$.
  If $G \in \Cc_2(\egp)$, then
  \begin{enumerate}[label=\rm(\roman*)]
  \item
    $\egp(G_1) = 1$,
  \item
    $\egp(G_2) = 2$.
  \end{enumerate}
\end{lemma}

\begin{proof}
  If $\egp(G_2) > 2$, then, since $G_2$ is a proper minor of $G$, there is a minor-operation $\mu \in \M(G_1)$,such that
  $\mu G$ still has $G_2$ as a minor. Hence $\egp(\mu G) \ge \egp(G_2) > 2$. We conclude that $\egp(G) = \egp(\mu G) = 3$, a contradiction. This shows that $\egp(G_2)\le2$.

  By Theorem~\ref{th-richter-euler}, we have
  $$3 = \egp(G) = \eh_1(G) = \egp(G_1) + \egp(G_2).$$
  Since $\egp(G_2) \le 2$, we conclude that $\egp(G_1) = 1$ and $\egp(G_2) = 2$.
  Thus (i) and (ii) hold.
\end{proof}

Finally, we are ready to state a theorem which classifies the $xy$-sums in $\Cc_2(\eg)$.

\begin{theorem}
\label{th-cc-2-eg}
  Let $G$ be the $xy$-sum of connected graphs $G_1, G_2 \in \Gcxy$.
  If the following statements {\rm(i)--\rm(iv)} hold, then $G \in \Cc_2(\eg)$.
  \begin{enumerate}[label=\rm(\roman*)]
  \item
    $G_1 \in \Cc_0(\egp)$.
  \item
    $G_2 \in \Cc_1(\egp) \cup \S_1$.
  \item
    If $G_1 \in \Cc_0(\eg)$, then $G_2 \in \Cc_1(\egp)$.
  \item
    If $G_1 \not\in \Cc_0(\eg)$, then $\et(G_2) \le 1$.
  \end{enumerate}
  Conversely, every 2-connected graph $G \in \Cc_2(\eg)$ such that $\{x,y\}$ is a 2-vertex-cut can be obtained in this way.
\end{theorem}

\begin{proof}
  Suppose that statements (i)--(iv) hold. Our goal is to show that $G \in \Cc_2(\eg)$.
  By Lemma~\ref{lm-tight}, it is enough to prove that $G_1$ and $G_2$ are $\eg$-tight in $G$ and that $\eg(G) = 3$.
  If $G_1 \in \Cc_0(\eg)$, then $\et(G_1) = 0$ by Lemma~\ref{lm-cc-0-prop}. Otherwise, $\et(G_1) = 1$ and $\et(G_2) \le 1$ by (iv).
  We conclude that in both cases we have $\eeta(G_1, G_2) \le 2$.
  Theorem~\ref{th-general-eg} and (i) give that $G_1$ is $\eg$-tight in $G$.
  If $\eeta(G_1, G_2) = 2$, then $G_2$ is $\eg$-tight in $G$ by Theorem~\ref{th-general-eg} and (ii).
  Suppose now that $\eeta(G_1, G_2) \le 1$ and $G_2 \in \S_1$.
  Since $\et(G_2) = 1$ by Lemma~\ref{lm-cascades}, we have that $\et(G_1) = 0$ and hence $G_1 \in \Cc_0(\eg)$ by Lemma~\ref{lm-cc-0-prop}.
  This is a contradiction with (iii).
  Thus, we may assume that $G_2 \in \Cc_1(\egp)$. Theorem~\ref{th-general-eg} asserts that $G_2$ is $\eg$-tight in $G$.
  Since $\eeta(G_1, G_2) \le 2$, $\egp(G_1) = 1$, and $\egp(G_2) = 2$, Theorem~\ref{th-richter-euler} and~(\ref{eq-eta-egp}) give that
  $$\eg(G) = \eh_1(G) = \egp(G_1) + \egp(G_2) = 3.$$
  Therefore, $G \in \Cc_2(\eg)$.

  We shall now show the converse, that is, for $G \in \Cc_2(\eg)$ where $\{x,y\}$ is a 2-vertex-cut, we find connected graphs $G_1, G_2 \in \Gcxy$ such that
  $G$ is an $xy$-sum of $G_1$ and $G_2$ and (i)--(iv) hold.
  Let us distribute the $\{x, y\}$-bridges arbitrarily into $G_1$ and $G_2$ so that $\egp(G_1) \le \egp(G_2)$
  and $G_1, G_2$ contain at least one of the bridges.
  By Lemma~\ref{lm-bounds-eg}, we have that $\egp(G_1) = 1$, $\egp(G_2) = 2$, and $\eeta(G_1, G_2) \le 2$.
  Since $\Hw_0(\eg)$ and $\S_0$ are empty (see Lemma~\ref{lm-cascades}),
  Theorem~\ref{th-general-eg} gives that $G_1 \in \Cc_0(\eg) \cup \Cc_0(\egp) = \Cc_0(\egp)$.
  Thus (i) holds.

  Since $\eeta(G_1, G_2) \le 2$, $G_2 \in \Cc(\eg) \cup \Cc(\egp) \cup \S \cup \Hw(\eg)$ by Theorem~\ref{th-general-eg}.
  By Lemma~\ref{lm-hoppers}, $\Hw_1(\eg)$ is empty.
  Since $\egp(G_2) = 2$, we have that $G_2 \not\in \Cc_0(\eg) \cup \Cc_0(\egp)$.
  We conclude that $G_2 \in \Cc_1(\eg) \cup \Cc_1(\egp) \cup \S_1$.
  Assume for a contradiction that $G_2 \in \Cc_1(\eg) \sm \Cc_1(\egp)$.
  Thus there exists a minor-operation $\mu \in \M(G_2)$ such that $\mu \not\in \dc2 \eg \cup \dc1 \egp$ since $\Hw_1(\eg)$ is empty. 
  By~(\ref{eq-eg-1}), $\eh_1(\mu G) = \eh_1(G)$. Since $G_2$ is $\eg$-tight in $G$, Theorem~\ref{th-richter-euler} gives:
  $$3 = \eg(G) > \eg(\mu G) = \eh_0(\mu G) = \eg(G_1) + \eg(\mu G_2) + 2 \ge 3.$$
  This contradicts our assumption that $G_2 \not\in \Cc_1(\egp) \cup \S_1$.
  We conclude that (ii) holds.

  Suppose that $G_1 \in \Cc_0(\eg)$ and $G_2 \in \S_1$.
  Since $\et(G_1) = 0$ and $\et(G_2) = 1$ by Lemmas~\ref{lm-cascades} and~\ref{lm-cc-0-prop}, we have that $\eeta(G_1, G_2) = 1$.
  This contradicts Theorem~\ref{th-general-eg}.
  Thus (iii) holds.

  In order to show (iv), suppose that $G_1 \not\in \Cc_0(\eg)$ and $\et(G_2) = 2$.
  Then $\et(G_1) = 1$ by Lemma~\ref{lm-cc-0-prop} and thus $\eeta(G_1, G_2) = 3$.
  This contradicts Lemma~\ref{lm-bounds-eg}(iii).
  We conclude that (iv) holds.
\end{proof}

We also have a corresponding theorem that classifies the $xy$-sums in $\Cc_2(\egp)$.

\begin{theorem}
\label{th-cc-2-egp}
  Let $G$ be the $xy$-sum of connected graphs $G_1, G_2 \in \Gcxy$.
  If the following statements (i) and (ii) hold, then $G \in \Cc_2(\egp)$.
  \begin{enumerate}[label=\rm(\roman*)]
  \item
    $G_1 \in \Cc_0(\egp)$.
  \item
    $G_2 \in \Cc_1(\egp)$.
  \end{enumerate}
  Conversely, every 2-connected graph $G \in \Cc_2(\egp)$ such that $\{x,y\}$ is a 2-vertex-cut can be obtained this way.
\end{theorem}

\begin{proof}
  Suppose that (i) and (ii) hold.
  By Theorem~\ref{th-general-egp}, $G_1$ and $G_2$ are $\egp$-tight in $G$.
  Thus $G \in \Cc(\egp)$ by Lemma~\ref{lm-tight}. By Theorem~\ref{th-richter-euler},
  $$\egp(G) = \eh_1(G) = \egp(G_1) + \egp(G_2) = 3.$$
  Therefore, $G \in \Cc_2(\egp)$.

  For the converse, let $G_1$ and $G_2$ be collections of $\{x, y\}$-bridges in $G$ such that $G$ is the $xy$-sum of $G_1$ and $G_2$, $\egp(G_1) \le \egp(G_2)$, and $G_1, G_2$ contain at least one of the bridges.
  We shall show that (i) and (ii) hold.
  By Theorem~\ref{th-general-egp}, $G_1, G_2 \in \Cc(\egp)$.
  By Lemma~\ref{lm-bounds-egp}, $\egp(G_1) = 1$ and $\egp(G_2) = 2$.
  We conclude that $G_1 \in \Cc_0(\egp)$ and $G_2 \in \Cc_1(\egp)$ and thus (i) and (ii) hold.
\end{proof}

The following lemma gives necessary and sufficient conditions for the edge $xy$ to be $\eg$-tight
in a graph with a 2-vertex-cut $\{x, y\}$ and the edge $xy$.

\begin{lemma}
\label{lm-edge-xy-eg}
  Let $G$ be an $xy$-sum of connected graphs $G_1, G_2 \in \Gcxy$ and let $H = \hat{G\+}$.
  Then the subgraph of $H$ consisting of the edge $xy$ is $\eg$-tight in $H$
  if and only if $\eeta(G_1, G_2) > 2$ and either $\eg(G_1 /xy) < \egp(G_1)$ or $\eg(G_2 / xy) < \egp(G_2)$,
\end{lemma}

\begin{proof}
  Since $\eg(H) = \egp(G) = \eg(G) + \et(G)$ and $\eg(H - xy) = \eg(G)$, we have that
  $\eg(H - xy) < \eg(H)$ if and only if $\et(G) > 0$.
  Theorem~\ref{th-richter-euler} gives that $\et(G) > 0$ if and only if $\eeta(G_1, G_2) > 2$.
  Thus we may assume below that $\eeta(G_1, G_2) > 2$.

  By Theorem~\ref{th-stahl-euler}, $\eg(H /xy) = \eg(G_1/xy) + \eg(G_2/xy)$.
  Since $\eg(G_1/xy) \le \egp(G_1)$ and $\eg(G_2 /xy) \le \egp(G_2)$, we have that
  $\eg(H/xy) < \eg(H)$ if and only if either $\eg(G_1/xy) < \egp(G_1)$ or $\eg(G_2/xy) < \egp(G_2)$.
\end{proof}

We conclude this section by characterizing the graphs of connectivity 2 in $\E_2$.

\begin{theorem}
  \label{th-e-2}
  Let $G$ be an $xy$-sum of connected graphs $G_1, G_2 \in \Gcxy$ such that the following holds:
  \begin{enumerate}[label=\rm(\roman*)]
  \item
    $G_1 \in \Cc_0(\egp)$.
  \item
    $G_2 \in \Cc_1(\egp) \cup \S_1$.
  \item
    If $G_1 \in \Cc_0(\eg)$, then $G_2 \in \Cc_1(\egp)$.
  \end{enumerate}
  If $\eeta(G_1, G_2) \le 2$, then $\hat{G} \in \E_2$.
  If $\eeta(G_1, G_2) > 2$, then $\hat{G\+} \in \E_2$.
  Furthermore, each graph in $\E_2$ of connectivity 2 is constructed this way.
\end{theorem}

\begin{proof}
  Assume first that $\eeta(G_1, G_2) \le 2$.
  By Lemma~\ref{lm-cc-eg}, it is sufficient to show that $G_1$ and $G_2$ satisfy the conditions (i)--(iv) of Theorem~\ref{th-cc-2-eg}.
  The conditions (i)--(iii) of Theorem~\ref{th-cc-2-eg} are the same as the assumptions of this theorem.
  If $G_1 \not\in \Cc_1(\eg)$, then $\et(G_1) = 1$ by Lemma~\ref{lm-cc-0-prop}.
  Since $\eeta(G_1, G_2) \le 2$, we have that $\et(G_2) \le 1$ and (iv) holds.
  By Theorem~\ref{th-cc-2-eg}, $G \in \Cc_2(\eg)$. By Lemma~\ref{lm-cc-eg}, $\hat{G} \in \E_2$.

  Assume now that $\eeta(G_1, G_2) > 2$.
  Since, for each graph $G \in \Cc_0(\egp) \cup \S_1$, $\et(G) \le 1$, by Lemmas~\ref{lm-cascades} and~\ref{lm-cc-0-prop},
  we conclude that $\eeta(G_1, G_2) = 3$, $\et(G_1) = 1$, $\et(G_2) = 2$, $G_1 \not\in \Cc_0(\eg)$, and $G_2 \in \Cc_1(\egp)$.
  By Theorem~\ref{th-cc-2-egp}, $G \in \Cc_2(\egp)$.
  Note that this implies that $\hat{G}$ is $\eg$-tight in $\hat{G\+}$.
  Since $\egp(G_1/xy) < \egp(G_1)$ (Lemma~\ref{lm-cc-0-prop}), we obtain that $xy$ is $\eg$-tight in $\hat{G\+}$ by Lemma~\ref{lm-edge-xy-eg}.
  Since $\eg(\hat{G\+}) = \egp(G) = 3$, $\hat{G\+} \in \E_2$ by Lemma~\ref{lm-tight}.

  Let us now prove that each $H \in \E_2$ of connectivity 2 is constructed this way.
  Pick an arbitrary 2-vertex-cut $\{x, y\}$ of $H$.
  Suppose first that $xy \in E(H)$.
  Consider $G = H-xy$ as a graph in $\Gcxy$.
  Since $\M(G) \ss \M(H)$, we have that $\M(G) = \dc1 \egp$ and $G \in \Cc(\egp)$.
  Suppose that $\egp(G) > 3$. Let $G_1$ and $G_2$ be parts of $G$ such that $\egp(G_1) \le \egp(G_2)$.
  If $\egp(G_2) > 2$, then for any minor-operation $\mu \in \M(G_1)$, the graph $\mu G$ has $G_2^+$ as a minor.
  Hence $\eg(\mu H) \ge \egp(G_2) > 2$, a contradiction. Therefore, $\egp(G_2) \le 2$.
  By Theorem~\ref{th-richter-euler}, $\egp(G_1) = \egp(G_2) = 2$.
  Let $\mu \in \M(G_1)$. By Theorem~\ref{th-richter-euler},
  $2 \ge \eg(\mu H) = \egp(\mu G) = \egp(\mu G_1) + \egp(G_2)$.
  Hence $\egp(\mu G_1) = 0$. We conclude that $\M(G_1) = \dc2 \egp$ and $G_1 \in \H_1(\egp)$.
  By Lemma~\ref{lm-hoppers}, $\H_1(\egp)$ is empty, a contradiction.

  So we may assume that $\egp(G) = 3$ and thus $G \in \Cc_2(\egp)$.
  By Theorem~\ref{th-cc-2-egp}, $G$ is an $xy$-sum of graphs $G_1 \in \Cc_0(\egp)$ and $G_2 \in \Cc_1(\egp)$.
  By Lemma~\ref{lm-edge-xy-eg}, $\eeta(G_1, G_2) > 2$.
  Thus $G$ satisfies the conditions (i)--(iii) of the theorem.

  Suppose now that $xy \not\in E(H)$. Consider $G = H$ as a graph in $\Gcxy$.
  By Lemma~\ref{lm-cc-eg}, $G \in \Cc(\eg)$.
  Suppose that $\eg(G) > 3$. Let $G_1$ and $G_2$ be parts of $G$ such that $\egp(G_1) \le \egp(G_2)$.
  If $\egp(G_2) > 2$, then for any minor-operation $\mu \in \M(G_1)$ so that $\mu G_1$ is connected, the graph $\mu G$ has $G_2^+$ as a minor.
  Hence $\eg(\mu H) \ge \egp(G_2) > 2$, a contradiction. Therefore, $\egp(G_2) \le 2$.
  By Theorem~\ref{th-richter-euler}, $\egp(G_1) = \egp(G_2) = 2$ and $3 < \eg(G) \le \eh_0(G) = \eg(G_1) + \eg(G_2) + 2$.
  We may assume that $\eg(G_1) \le \eg(G_2)$ and so $\eg(G_2) \ge 1$.
  Let $\mu \in \M(G_1)$. By Theorem~\ref{th-richter-euler},
  $$2 \ge \eg(\mu H) = \eg(\mu G) = \min\{\eh_0(\mu G), \eh_1(\mu G)\}.$$
  Since $\eh_0(\mu G) = \eg(\mu G_1) + \eg(G_2) + 2 \ge 3$, we have that
  $2 \ge \eh_1(\mu G) = \egp(\mu G_1) + \egp(G_2)$.
  We conclude that $\egp(\mu G_1) = 0$ and $\mu \in \dc2 \egp$. Since $\mu$ was arbitrary, $G_1 \in \H_1(\egp)$.
  This contradicts Lemma~\ref{lm-hoppers} which asserts that $\H_1(\egp)$ is empty.

  Thus we may assume that $\eg(G) = 3$ and thus $G \in \Cc_2(\eg)$.
  By Theorem~\ref{th-cc-2-eg}, $G$ is an $xy$-sum of graphs $G_1 \in \Cc_0(\egp)$ and $G_2 \in \Cc_1(\egp) \cup \S_1$
  and either $G_1 \not\in \Cc_0(\eg)$ or $G_2 \not\in \S_1$.
  If $G_1 \in \Cc_0(\eg)$, then $\et(G_1) = 0$ by Lemma~\ref{lm-cc-0-prop} and thus $\eeta(G_1, G_2) \le 2$.
  Otherwise, $\et(G_1) \le 1$ and $\et(G_2) \le 1$ by Theorem~\ref{th-cc-2-eg}(iv) and we obtain that $\eeta(G_1, G_2) \le 2$.
  Thus $G$ satisfies the conditions (i)--(iii) of the theorem.
\end{proof}

As a corollary we can construct the complete list of graphs in $\E_2$ of connectivity~2.
The completeness of the list relies on the classification of cascades in Part II \cite{MS2014_II} of our work, where it is shown that $\S_1$ contains precisely 21 cascades.

\begin{corollary}
  \label{cr-e-2}
  There are precisely $668$ graphs of connectivity 2 that are critical for Euler genus 2. 
\end{corollary}

\begin{proof}
  Let us begin by counting the number of pairs $G_1, G_2$ that satisfy the conditions (i)--(iii) of Theorem~\ref{th-e-2}.
  There are 3 graphs in $\Cc_0(\egp)$, there are 227 graphs $G_2$ in $\Cc_1(\egp)$ such that $G_2^+$ is 2-connected, and there are $21$ graphs in $\S_1$ (for each $G \in \S_1$, the graph $G^+$ is 2-connected). 
  That gives $744$ pairs since $|\Cc_0(\eg)|=1$ and $|\S_1|=21$. 
  There are only $744-21=723$ pairs that satisfy the condition (iii) of Theorem~\ref{th-e-2} that either $G_1 \not\in \Cc_0(\eg)$ or $G_2 \not\in \S_1$.  

  Let $G_1, G_2 \in \Gcxy$.
  There are two $xy$-sums that have parts isomorphic to $G_1$ and $G_2$ as there are two ways how to identify two pairs of vertices.
  If $G_1 \in \Cc_0(\egp)$, then there is an automorphism of $G_1$ exchanging the terminals. Hence there is only a single non-isomorphic $xy$-sum $G$
  that has parts $G_1 \in \Cc_0(\egp)$ and $G_2 \in \Cc_1(\egp) \cup \S_1$.
  Since $\eeta(G_1, G_2)$ depends only on $G_1$ and $G_2$, precisely one of $\hat{G}, \hat{G\+}$ belongs to $\E_2$.
  There may be more pairs $G_1, G_2$ giving the same graph $H \in \E_2$ though.

  Let $H \in \E_2$ have connectivity~2. By Theorem~\ref{th-e-2}, there exists an $xy$-sum $G$ of connected graphs $G_1$ and $G_2$
  such that either $\hat{G} \iso H$ or $\hat{G\+} \iso H$.
  Note that $G_1^+$ and $G_2^+$ are 2-connected.
  Suppose that $H$ admits a nontrivial automorphism $\psi$ such that $\psi(V(G_1)) \not= V(G_1)$ (otherwise, it is just a combination of two automorphisms of $G_1$ and $G_2$).
  It is not hard to see that if $G_2^+$ is 3-connected, each automorphism of $H$ is trivial.
  Therefore, we need to study graphs $G_2 \in \Cc_1(\egp) \cup \S_1$ such that $G_2^+$ has connectivity~2.

  There are 39 graphs $G_2$ in $\Cc_1(\egp)$ such that $G_2^+$ has connectivity 2 and there are 4 graphs $G_2$ in $\S_1$ such that $G_2^+$ has connectivity 2 (see Part II \cite[Lemma 7.10]{MS2014_II} for details). 
  It is not hard to check that the $125$ pairs 
  with $G_1 \in \Cc_0(\egp)$ make only 70 non-isomorphic graphs in $\E_2$. 
  We conclude that there are $668$ graphs of connectivity 2 in $\E_2$. 
\end{proof}

\section{The Klein bottle}
\label{sc-klein-klein}

In this section, we characterize the obstructions of connectivity 2 for embedding graphs into the Klein bottle. We show in fact that they are the same as the 668 critical graphs for Euler genus 2 (Corollary \ref{cr-e-2}).

Let us introduce graph parameters $\os$ and $\osp$ that capture the property of being orientably simple.
Let $\os = \ng - \eg$ and let $\osp = \ngp - \egp$.
Note that $\os(G) = 1$ if $G$ is orientably simple and $\os(G) = 0$ otherwise.

The following lemma is an easy consequence of Lemma~\ref{lm-ng-upper}.

\begin{lemma}
\label{lm-non-simple}
If\/ $\egp(G)$ is odd, then $\osp(G) = 0$.
\end{lemma}

Let us state the following theorem of Stahl and Beineke using our formalism.

\begin{theorem}[Stahl and Beineke~\cite{stahl-1977}]
\label{th-stahl-nonori}
  Let $G = G_1 \cup G_2$ be a $1$-sum of $G_1$ and $G_2$. Then
  $$\ng(G) = \eg(G_1) + \eg(G_2) + \os(G_1)\os(G_2).$$
  Moreover, $\os(G) = \os(G_1)\os(G_2)$.
\end{theorem}

In order to describe how the nonorientable genus of a 2-sum of graphs can be computed from
the genera of its parts, let us introduce parameters $\nh_0$ and $\nh_1$ similar to $\eh_0$
and $\eh_1$.
Let $G$ be an $xy$-sum of connected graphs $G_1, G_2 \in \Gxy$.
Define
\begin{equation}
  \nh_0(G) = \eh_0(G) = \eg(G_1) + \eg(G_2) + 2\label{eq-ng-0}
\end{equation}
and
\begin{equation}
  \nh_1(G) = \egp(G_1) + \egp(G_2) + \osp(G_1)\osp(G_2).\label{eq-ng-1}
\end{equation}

Let $\nt = \ngp - \ng$.
We shall use the following theorem of Richter.

\begin{theorem}[Richter~\cite{richter-1987-nonori}]
  \label{th-richter-nonori}
  Let $G$ be an $xy$-sum of connected graphs $G_1, G_2 \in \Gxy$.
  Then
  \begin{enumerate}[label=\rm(\roman*)]
  \item
    $\ng(G) = \min\{\nh_0(G), \nh_1(G)\}$,
  \item
    $\ngp(G) = \nh_1(G)$,
  \item
    $\nt(G) = \max\{\nh_1(G) - \nh_0(G), 0\}$,
  \item
    $\osp(G) = \osp(G_1)\osp(G_2)$, and
  \item
    if $\osp(G) = 0$ or $\eeta(G_1, G_2) \ge 2$, then $\os(G) = 0$, else $\os(G) = 1$.
  \end{enumerate}
\end{theorem}

The next lemma shows that the $xy$-sums of graphs with parts that are not orientably simple are
critical graphs for Euler genus if and only if they are obstructions for the corresponding nonorientable surface.

\begin{lemma}
  \label{lm-osp-0}
  Let $G$ be an $xy$-sum of connected graphs $G_1, G_2 \in \Gcxy$, $H \in \{\hat{G}, \hat{G\+}\}$, and $k \ge 0$.
  If $\osp(G_1) = \osp(G_2) = 0$, then $H \in \E_k$ if and only if $H \in Forb(\NN_k)$.
\end{lemma}

\begin{proof}
  By Theorem~\ref{th-richter-nonori}(iv) and (v), $\os(G) = \osp(G) = \osp(G_1)\osp(G_2) = 0$.
  Therefore, $\os(H) = 0$.

  Assume first that $H \in \Forb(\NN_k)$.
  We have that $\eg(H) = \ng(H) > k$.
  Let $\mu \in \M(H)$. Since $\eg(\mu H) \le \ng(\mu H) \le k$ and $\mu$ is arbitrary, we have that $H \in \E_k$.

  Assume now that $H \in \E_k$.
  By Lemmas~\ref{lm-cc-eg} and~\ref{lm-cc-egp-iff}, $G \in \Cc(\eg) \cup \Cc(\egp)$.
  We have that $\ng(H) = \eg(H) > k$.
  Let $\mu \in \M(G_1)$. By Lemma~\ref{lm-no-cutedges}, $\mu G_1$ is connected.
  By Theorem~\ref{th-richter-nonori}(iv) and (v), $\os(\mu G) = \osp(\mu G) = \osp(\mu G_1)\osp(G_2) = 0$.
  Therefore, $\os(\mu H) = 0$.
  Hence $\ng(\mu H) = \eg(\mu H) \le k$.
  Similarly $\ng(\mu H) \le k$ for $\mu \in \M(G_2)$.
  This shows that $H \in \Forb(\NN_k)$ if $H = \hat{G}$.
  Assume then that $H = \hat{G\+}$. It remains to see that after deleting or contracting
  the edge $xy$, the graph can be embedded in $\NN_k$.
  Since $\os(H - xy) = \os(G) = 0$, we have that $\ng(H - xy) = \eg(H - xy) \le k$.
  By Theorem~\ref{th-stahl-nonori},
  $\os(H / xy) = \os(G_1/xy)\os(G_2/xy)$.
  If $\os(H /xy) = 0$, then
  $\ng(H /xy) = \eg(H /xy) \le k$.
  So, we are done unless $\os(G_1/xy) = \os(G_2/xy) = 1$, which we assume henceforth.
  Since $G_1/xy$ is a minor of $\hat{G_1^+}$, we have that
  $$\eg(G_1/xy) = \ng(G_1/xy) - \os(G_1/xy) < \ng(G_1/xy) \le \ng(G_1^+) = \ngp(G_1) = \egp(G_1).$$
  Similarly, we derive that $\eg(G_2/xy) < \egp(G_2)$.
  By Theorems~\ref{th-richter-euler} and~\ref{th-stahl-nonori},
  \begin{align*}
    \ng(H /xy) &\le \eg(G_1/xy) + \eg(G_2/xy) + 1 < \egp(G_1) + \egp(G_2) \\
    &= \egp(G) = \eg(H) = \ng(H) \le k.
  \end{align*}
  We conclude that $H \in \Forb(\NN_k)$.
\end{proof}

A corollary of Theorem~\ref{th-e-2} and Lemma~\ref{lm-osp-0} asserts that the class of obstructions for the Klein bottle having connectivity 2 and the class of critical graphs for Euler genus 2 of connectivity 2 are the same.
We can say even more:

\begin{corollary}
  \label{cr-klein-bottle}
  Let $H$ be a graph of connectivity 2. Then $H \in \E_2$ if and only if $H \in \Forb(\NN_2)$.
\end{corollary}

\begin{proof}
  Assume first that $H \in \E_2$.
  By Theorem~\ref{th-e-2}, there is an $xy$-sum $G$ of graphs $G_1 \in \Cc_0(\egp)$ and $G_2 \in \Cc_1(\egp) \cup \S_1$
  such that $H \in \{\hat{G}, \hat{G\+}\}$.
  By Lemma~\ref{lm-non-simple}, $\osp(G_1) = 0$.
  By Lemma~\ref{lm-all-in-klein}, $\osp(G_2) = 0$.
  By Lemma~\ref{lm-osp-0}, $H \in \Forb(\NN_2)$.

  Assume now that $H \in \Forb(\NN_2)$.
  Let $G$ be an $xy$-sum of connected graphs $G_1, G_2 \in \Gcxy$ such that $H \in \{\hat{G}, \hat{G\+}\}$.
  Suppose that $\osp(G_1) = 1$.
  If $\egp(G_1) \ge 2$, then $\ngp(G_1) \ge 3$ and thus $G_1^+$ does not embed into $\NN_2$.
  This yields a contradiction as $H$ has $G_1^+$ as a proper minor.
  Since $\egp(G_1) > 0$, we conclude that $\egp(G_1) = 1$.
  By Lemma~\ref{lm-non-simple}, $\osp(G_1) = 0$, a contradiction.
  Therefore by symmetry, $\osp(G_1) = \osp(G_2) = 0$.
  By Lemma~\ref{lm-osp-0}, $H \in \E_2$.
\end{proof}

\bibliographystyle{abbrv}
\bibliography{bibliography}

\end{document}